\documentclass{amsart}
\usepackage[utf8]{inputenc}
\usepackage{amssymb,amsmath,amsthm,verbatim,color}
\usepackage[all]{xy}
\def\bP{\mathbb{P}}
\def\cO{\mathcal{O}}
\def\cI{\mathcal I}
\def\wt{\widetilde}

\def\coker{{\rm coker}}
\def\Cliff{{\rm Cliff}}
\def\Ext{\mathrm{Ext}}
\def\E{\mathcal{E}}
\newtheorem{thm}{Theorem} 
\newtheorem{prop}[thm]{Proposition}

\newtheorem{cor}[thm]{Corollary}

\newtheorem{rem}[thm]{Remarks}
\begin{document}
\title{Non-surjective Gaussian maps for singular curves on K3 surfaces}
\author{Claudio Fontanari and Edoardo Sernesi}
\begin{abstract} Let $(S,L)$ be a polarized K3 surface with $\mathrm{Pic}(S) = \mathbb{Z}[L]$ and 
$L\cdot L=2g-2$, let $C$ be a nonsingular curve of genus $g-1$ and let $f:C\to S$ be such that $f(C) \in \vert L \vert$. We prove that the Gaussian map $\Phi_{\omega_C(-T)}$ is non-surjective, where $T$ is the degree two divisor over the singular point $x$ of $f(C)$. This generalizes a result of Kemeny with an entirely different proof. It uses the very ampleness of $C$ on the blown-up surface $\wt S$ of $S$ at $x$ and a theorem of L'vovski.   

\end{abstract}

\address{Dipartimento di Matematica, Universit\`a di Trento, Via Sommarive 14, 38123 Trento, Italy.}
\email{claudio.fontanari@unitn.it}
\address{Dipartimento di Matematica e Fisica, Universit\`a Roma Tre, L.go S.L. Murialdo 1, 00146 Roma, Italy.} 
\email{sernesi@gmail.com}
\thanks{This research has been partially supported by GNSAGA of INdAM, by PRIN 2015 ``Geometria delle variet\`a
algebriche'', and by FIRB 2012 ``Moduli spaces and Applications''.}
\date{}
\maketitle

\section{Introduction}  

Let $C$ be a complex projective nonsingular curve of genus $g$. Let $L,A$  be invertible sheaves on $C$ and let 
$$
\mathcal{R}(L,A):= \ker[H^0(C,L)\otimes H^0(C,A) \longrightarrow H^0(C,LA)]
$$
Then we can define a \emph{Gaussian map} 
$$
\Phi_{L,A}: \mathcal{R}(L,A) \longrightarrow H^0(C,\omega_CLA)
$$
in a well-known way that will be recalled in \S \ref{S:conormal}.  
If $L=A$ the map $\Phi_{L,L}$ has the same image as its restriction to $\bigwedge^2 H^0(C,L)\subset \mathcal{R}(L,L)$, which is denoted by:
$$
\Phi_L: \bigwedge^2 H^0(C,L) \longrightarrow H^0(C,\omega_CL^2)
$$
If we take $L=\omega_C$ then the  map:
$$
\Phi_{\omega_C}:\bigwedge^2 H^0(C,\omega_C) \longrightarrow H^0(C,\omega_C^3)
$$
is called the \emph{Wahl map}.  

The following result, due to Wahl (see \cite{jW87} and also \cite{BM87} for a different proof), gives a necessary condition for a nonsingular curve to be hyperplane section of a K3 surface:

\begin{thm}[Wahl]\label{wahl}
Every nonsingular curve in a very ample linear system $\vert L \vert$ 
on a K3 surface $S$ has non-surjective Wahl map.
\end{thm}

This result has been generalized by L'vovski (see \cite{Lv92} and also \cite{BF03} for an elementary detailed proof) in the following form:

\begin{thm}[L'vovski]\label{l'vovski} Let $C$ be a smooth curve of genus $g >0$ and let $A$ be a very ample line bundle on $C$ embedding $C$ in $\mathbb{P}^n$, $n \ge 3$. If $C \subset \mathbb{P}^n$ is scheme-theoretically a hyperplane section of a smooth surface $X \subset \mathbb{P}^{n+1}$ then the Gaussian map $\Phi_{\omega_C, A}$ is non-surjective. 
\end{thm}

In this paper we will focus on singular curves on a K3 surface $S$. Our starting point is the   recent work \cite{mK15}  by Kemeny. Let $\mathcal{V}^1_g$ be the moduli space of triples $(S,L, f: C \to S)$, where $(S,L)$ is a polarized K3 surface with $L \cdot L =2g-2$, $C$ is a smooth curve of genus $g-1$ and $f$ is an unramified stable map, birational onto its image, such that $f(C)\in \vert L\vert$.  Then the following holds:  
 
\begin{thm}[\cite{mK15}, Theorem 1.7]\label{kemeny} Fix an integer $g \ge 14$. Then there is an irreducible component $I^0 \subseteq \mathcal{V}^1_g$ such that for a general triple $(S,L, f: C \to S) \in I^0$ the Gaussian map $\Phi_{\omega_C(-T)}$ is non-surjective, where $T=P+Q \subseteq C$ is the divisor over the node of $f(C)$. 
\end{thm}

The component $I^0$ appearing in the statement might a priori include all 1-nodal curves in $\vert L\vert$, but this is not proved in \cite{mK15}. The  proof is rather indirect and relies on the fact that $H^0(C,f^*T_S)=0$ for the general triple $(S,L, f: C \to S) \in I^0$ (see \cite{mK15}, Lemma 3.17).
 
Our  main result is the following more general statement, which is an exact analogue of Theorem \ref{wahl} for singular curves:

\begin{thm}\label{main} Fix an integer $g \ge 9$. Let $(S,L)$ be a polarized K3 surface such that $\mathrm{Pic}(S) = \mathbb{Z}[L]$ and $L \cdot L =2g-2$. Let $C$ be a smooth curve of genus $g-1$ endowed with a morphism $f: C \to S$ birational onto its image and such that $f(C) \in \vert L \vert$. If  $T=P+Q \subseteq C$ is the divisor over the singular point of $f(C)$, then the Gaussian map $\Phi_{\omega_C(-T)}$ is non-surjective.  
\end{thm}

Note that we are not making any generality assumption on the triple $(S,L, f: C \to S)$. In particular we are not assuming that $f(C)$ is a nodal curve: the hypothesis that $C$ has genus $g-1$ just implies that $f(C)$ is either 1-nodal or has an ordinary cusp. Note also, by contrast, that the normalization $C$ of a 1-nodal curve on a K3 surface tends to have \emph{surjective} Wahl map $\Phi_{\omega_C}$, and therefore, by Theorem \ref{wahl}, not to be embeddable in any K3 surface. In fact, the following result holds:
 
\begin{thm}[Sernesi \cite{eS17}] Let $(S,L)$ be a general primitively polarized K3 surface of genus $g+1$. Assume that $g=40$, $42$ or $\ge 44$. Then the Wahl map of the normalization of any 
 $1$-nodal curve in $\vert L \vert$ is surjective.
\end{thm}

In outline, the proof of Theorem \ref{main} goes as follows. 
 
We prove that  on the blow-up $\sigma: \wt S \to S$ at the singular point of $f(C)$ the line bundle $H := \sigma^*L(-2E)$ is very ample, where $E$ denotes the exceptional divisor of the blow-up $\sigma: \wt S \to S$. This fact is a special case of Theorem \ref{veryample}, which gives a more general very-ampleness criterion of independent interest.
 
Hence we can apply Theorem \ref{l'vovski} to 
$$
A := H \cdot C = C \cdot C = (C+E) \cdot C - E \cdot C = \omega_C(-T)
$$
and obtain that the Gaussian map $\Phi_{\omega_C,\omega_C(-T)}$ is non-surjective. 

Finally, we prove that $\mathrm{coker}(\Phi_{\omega_C(-T)})$ surjects onto 
$\mathrm{coker}(\Phi_{\omega_C,\omega_C(-T)})$ (Theorem \ref{comparison}).  
In particular, also $\Phi_{\omega_C(-T)}$ is non-surjective.

We work over the field $\mathbb{C}$ of complex numbers.


\section{Conormal sheaves and projections}\label{S:conormal}

Let $C$ be as in the Introduction.  It turns out to be  natural, for our purposes, to introduce the so-called syzygy sheaves. We define the \emph{syzygy sheaf} $M_L$ of a globally generated invertible sheaf $L$ by the exact sequence:
$$
0\to M_L \to H^0(C,L)\otimes \cO_C \to L \to 0
$$
If $L$ is very ample the above sequence is a twist of the dualized Euler sequence, and we get  
\begin{equation}\label{E:con1}
M_L=\Omega^1_{\bP|C}\otimes L
\end{equation}
where $\bP\cong\bP H^0(C,L)^\vee$. Therefore the conormal sequence of $C\subset \bP$ twisted by $L$ takes the following form:
\begin{equation}\label{conormalL}
\xymatrix{
0\ar[r]&N^\vee_{C/\bP}\otimes L\ar[r]&M_L\ar[r]&\omega_CL\ar[r]&0
}
\end{equation}
Now let $A$ be another invertible sheaf on $C$ and tensor \eqref{conormalL} by $A$:
\begin{equation}\label{conormalLA}
\xymatrix{
0\ar[r]&N^\vee_{C/\bP}\otimes LA\ar[r]&M_L\otimes A\ar[r]^-\rho&\omega_CLA\ar[r]&0
}
\end{equation}
Then $H^0(C,M_L\otimes A)=\mathcal{R}(L,A)$ and the map induced by $\rho$ on global sections:
$$
\Phi_{L,A}: H^0(C,M_L\otimes A)\longrightarrow H^0(C,\omega_CLA)
$$
is the \emph{Gaussian map} of $L,A$. When $L=A$ we have:
$$
H^0(C,M_L\otimes L)= \mathbf{I}_2 \bigoplus \bigwedge^2 H^0(C,L)
$$
where 
$$
\mathbf{I}_2 = \ker [S^2H^0(C,L) \to H^0(C,L^2)]
$$
Since $\mathbf{I}_2\subset \ker(\Phi_{L,L})$ the  map $\Phi_{L,L}$ has the same image as its restriction to $\bigwedge^2 H^0(C,L)\subset \mathcal{R}(L,L)$, which is denoted by:
$$
\Phi_L: \bigwedge^2 H^0(C,L) \longrightarrow H^0(C,\omega_CL^2)
$$

\medskip

 Now let $L$ be very ample, $P\in C$ and assume that $L(-P)$ is also very ample. Then  we have embeddings:
$$
\varphi_L: C \to \bP^r, \quad \varphi_{L(-P)}: C \to \bP^{r-1}
$$
where $h^0(C,L)=r+1$. The following proposition relates the conormal sheaves $N^\vee_{C/\bP^r}$ and $N^\vee_{C/\bP^{r-1}}$.

\begin{prop}\label{P:con1}
There is an exact sequence:
\begin{equation}\label{E:con2}
0 \to N^\vee_{C/\bP^{r-1}}\otimes L(-P) \to N^\vee_{C/\bP^r}\otimes L \to \cO_C(-2P) \to 0
\end{equation}
\end{prop}

\begin{proof}
There is an exact sequence 
$$
0\to M_{L(-P)} \to M_L \to \cO_C(-P) \to 0
$$
induced by the inclusion $H^0(C,L(-P)) \subset H^0(C,L)$. Recalling \eqref{E:con1} we get a commutative and exact diagram whose first two rows are  twisted conormal sequences:
$$
\xymatrix{&0\ar[d]&0\ar[d]&0\ar[d]\\
0\ar[r]&N^\vee_{C/\bP^{r-1}}\otimes L(-P)\ar[d]\ar[r]&M_{L(-P)}\ar[r]\ar[d]&\omega_CL(-P)\ar[d]\ar[r]&0\\
0\ar[r]&N^\vee_{C/\bP^r}\otimes L\ar[r]\ar[d]&M_L\ar[d]\ar[r]&\omega_CL\ar[r]\ar[d]&0\\
0\ar[r]&\cO_C(-2P)\ar[d]\ar[r]&\cO_C(-P)\ar[d]\ar[r]&\omega_CL\otimes \cO_P\ar[d]\ar[r]&0\\
&0&0&0
}
$$
The first column gives the sequence \eqref{E:con2}.
\end{proof}

\begin{cor}
Let $C,L$ be as before and suppose that  $P,Q\in C$  are points such that $L(-P-Q)$ is very ample. Then there is an exact sequence:
\begin{equation}\label{E:con3}
0 \to N^\vee_{C/\bP^{r-2}}\otimes L(-P-Q)\to  N^\vee_{C/\bP^r}\otimes L \to \cO_C(-2P)\bigoplus\cO_C(-2Q) \to 0
\end{equation}
\end{cor}

\begin{proof}
Left to the reader.
\end{proof}


\section{A comparison result between Gaussian maps}

After the preliminaries collected in the previous section, we are ready to prove the following result:

\begin{thm}\label{comparison}
Let $C$ be a projective nonsingular curve of genus $g$, $T=P+Q$ an effective divisor of degree $2$ on $C$. Assume that $\Cliff(C)\ge 3$. Then there is a surjection:
  $$
   \coker(\Phi_{\omega_C(-T)}) \longrightarrow \coker(\Phi_{\omega_C,\omega_C(-T)})\longrightarrow 0
  $$
In particular, if $\Phi_{\omega_C,\omega_C(-T)}$ is not surjective then $\Phi_{\omega_C(-T)}$ is not surjective.
\end{thm}

\begin{proof}
The hypothesis $\Cliff(C)\ge 3$ implies  that $\omega(-T)$ is very ample and maps $C\subset \bP^{g-3}$. We have an exact sequence  (\cite{rL89}, Lemma 1.4.1):
$$
0 \to M_{\omega_C(-T)} \to M_{\omega_C} \to \cO_C(-P)\bigoplus \cO_C(-Q) \to 0
$$
which, twisted by $\omega_C(-T)$,  appears as the middle column in the following diagram:
$$\begin{small}
\xymatrix{&0\ar[d]&0\ar[d]&0\ar[d]\\
0\ar[r]&N^\vee_{C/\bP^{g-3}}\otimes \omega^2_C(-2T)\ar[d]\ar[r]&M_{\omega_C(-T)}\otimes\omega_C(-T)\ar[r]^-a\ar[d]^-f&\omega_C^3(-2T)\ar[d]\ar[r]&0\\
0\ar[r]&N^\vee_{C/\bP^{g-1}}\otimes \omega^2_C(-T)\ar[d]\ar[r]&M_{\omega_C}\otimes \omega_C(-T)\ar[d]\ar[r]^-b&\omega_C^3(-T)\ar[d]\ar[r]&0\\
0\ar[r]&\omega_C(-T-2P)\bigoplus \omega_C(-T-2Q)\ar[d]\ar[r]^-g&\omega_C(-T-P)\bigoplus \omega_C(-T-Q)\ar[d] \ar[r]&\omega_C^3(-T)\otimes\cO_T\ar[r]\ar[d]&0\\
&0&0&0}
\end{small}$$
where the first column is \eqref{E:con3} for $L=\omega_C$, twisted by $\omega_C(-T)$.  The homomorphisms $a$ and $b$ induce $\Phi_{\omega_C(-T),\omega_C(-T)}$ and $\Phi_{\omega_C,\omega_C(-T)}$ respectively on global sections.  Therefore, taking cohomology,  we obtain the following diagram:
$$
\begin{small}
\xymatrix{
0\ar[r]&\coker(\Phi_{\omega_C(-T)})\ar[r]\ar[d]^-\zeta&H^1(N^\vee_{C/\bP^{g-3}}\otimes \omega^2_C(-2T))\ar[r]\ar[d]&H^1(\Omega^1_{\bP^{g-3}|C}\otimes\omega^2_C(-2T))\ar[r]\ar[d]^-{H^1(f)}&0\\
0\ar[r]&\coker(\Phi_{\omega_C,\omega_C(-T)})\ar[d]\ar[r]&H^1(N^\vee_{C/\bP^{g-1}}\otimes \omega^2_C(-T))\ar[d]\ar[r]&H^1(\Omega^1_{\bP^{g-1}|C}\otimes \omega^2_C(-T))\ar[d]\ar[r]&0\\
0\ar[r]&\ker(H^1(g))\ar[r]&H^1(\omega_C(-T-2P))\atop{\bigoplus\atop H^1(\omega_C(-T-2Q))}\ar[r]^-{H^1(g)}\ar[d]&H^1(\omega_C(-T-P))\atop{\bigoplus \atop H^1(\omega_C(-T-Q))}\ar[r]\ar[d]&0\\
&&0&0
}
\end{small}
$$
Since $\Cliff(C) \ge 3$ it follows that $H^1(g)$ is an isomorphism, thus $\zeta$ is surjective. 
\end{proof}

Theorem \ref{comparison} can be generalized in several ways. for example, using similar methods   and   induction on $n$ one can also prove the following:

\begin{thm}
Let $C$ be a projective nonsingular curve of genus $g$, $T=P_1+\cdots + P_n$ an effective divisor of degree $n \ge 1$. Let $L$ be an invertible sheaf on $C$ of degree $d \ge 2g+1+n$. Then there is a surjection:
$$
\coker(\Phi_{L(-T)}) \longrightarrow \coker(\Phi_{L,L(-T)})\longrightarrow 0
  $$
\end{thm}

We will not pursue this here.

\section{Very ampleness on blown-up surfaces}

For the proof of Theorem \ref{main} we will need a special case of the following result of independent interest:

\begin{thm}\label{veryample}
Let $S$ be a K3 surface such that Pic$(S)=\mathbb{Z}[L]$ for some ample invertible sheaf $L$. Assume that $L\cdot L=2g-2\ge (\ell+1)^2+3$ for some $\ell\ge 1$. Let  $x\in S$,   $\sigma: \wt S \to S$ the blow-up of $S$ at $x$ and $E\subset \wt S$ the exceptional curve. Then the sheaf $H:=\sigma^*L(-\ell E)$ is very ample on $\wt S$.
\end{thm}

\begin{proof}
We follow closely the application of Reider's method in \cite{cV17}, proof of Theorem 16. We must prove that for each subscheme $Z\subset \wt S$ of length two we have $H^1(\wt S,\cI_Z\otimes H)=0$.

By contradiction, assume that $H^1(\wt S,\cI_Z\otimes H)\ne 0$ for some $Z$. By Serre duality we have:
$$
H^1(\wt S,\cI_Z\otimes H)^{\vee}= \Ext^1(\cI_Z,E-H)
$$
and therefore there is a non-split exact sequence:
\begin{equation}\label{E:extn1}
\xymatrix{
0\ar[r]&\sigma^*L^{-1}((\ell+1)E)\ar[r]& \E \ar[r]& \cI_Z \ar[r]&0
}
\end{equation}
where $\E$ is torsion free. 
We have: 
\begin{align*}
&c_1(\E)^2= 2g-2-(\ell+1)^2, \\  &c_2(\E)=2 \\
&\chi(Hom(\E,\E)) = 4 \chi(\cO_{\wt S})+c_1^2(\E)-4 c_2(\E) = 2g-2-(\ell+1)^2 \ge 3
\end{align*}
Therefore by Serre duality:
$$
2h^0(Hom(\E,\E(E))) \ge h^0(Hom(\E,\E))+ h^0(Hom(\E,\E(E)) \ge 3
$$
It follows that there is a homomorphism $\phi: \E \longrightarrow \E(E)$ which is not proportional to the identity.  By the usual trick we can assume that $\phi$ is generically of rank one (see \cite{cV17}, proof of Proposition 15).    $\ker(\phi)$ and im$(\phi)$ are torsion free rank one, therefore of the form $A\otimes \cI_W$ and $B\otimes\cI_{W'}$ respectively, for some invertible sheaves $A,B$ which are of the form:
$$
A = \sigma^*L^\alpha(\beta E), \quad B= \sigma^*L^{-1-\alpha}((\ell+1-\beta)E)
$$
From the exact sequence
  $$
  \xymatrix{
  0\ar[r]& A\otimes\cI_W \ar[r]&\E\ar[r]& B\otimes \cI_{W'}\ar[r]&0}
 $$
 we compute:
 \begin{equation}\label{E:extn2}
 2 = c_2(\E)= \deg(W)+\deg(W')- \beta(\ell+1-\beta) \ge - \beta(\ell+1-\beta) 
 \end{equation}
 Indeed, since $A\otimes \cI_W\subset \E$, from \eqref{E:extn1} we see that we must have $\alpha \le 0$. Similarly $-\alpha-1 \le 0$ because $B\otimes\cI_{W'} \subset \E(E)$.  Therefore:
 $$
 -1 \le \alpha \le 0,
 $$
 proving (\ref{E:extn2}).
 \medskip
 
  Suppose $\alpha=0$. Then we have an inclusion $\cO_{\wt S}(\beta E)\otimes \cI_W \subset \cI_Z$, which implies $\beta \le 0$.  If $\beta < 0$ then 
 \eqref{E:extn2} gives a contradiction. If $\beta=0$ we get an inclusion $\cI_W\subset \cI_Z$. This implies that the pullback homomorphism:
 $$
 \psi:\Ext^1(\cI_Z,E-H) \longrightarrow \Ext^1(\cI_W,E-H)
 $$
 maps \eqref{E:extn1} to zero, thus $\psi$ is not injective.  But $\psi$ is dual to:
 $$
 \psi^\vee:H^1(\wt S, \cI_W\otimes H) \longrightarrow H^1(\wt S, \cI_Z\otimes H)
 $$
 which is henceforth not surjective. On the other hand   the diagram:
 \begin{equation}\label{E:extn3}
 \xymatrix{
 &0\ar[d]&&\mathcal{T}\ar[d]\\
0\ar[r]&\cI_W\otimes H \ar[d]\ar[r]&H\ar@{=}[d]\ar[r]&H_{|W}\ar[d]\ar[r]&0\\
 0\ar[r]&\cI_Z\otimes H \ar[r]&H\ar[r]&H_{|Z}\ar[r]\ar[d]&0\\
 &&&0}
 \end{equation}
 shows that we have an exact sequence
 $$
 \xymatrix{
 0\ar[r]&\cI_W\otimes H\ar[r]&\cI_Z\otimes H \ar[r]& \mathcal{T} \ar[r]&0
 }
 $$
 with $\mathcal{T}$ torsion: therefore $\psi^\vee$ is surjective, and we have a contradiction. So the case $\alpha=0$ cannot occur.
 
 \medskip
 
 Suppose $\alpha=-1$. In this case we use the inclusion $B\otimes \cI_{W'}\subset \E(E)$.  We obtain an inclusion  $\cI_{W'}((\ell-\beta)E)\subset \E$ which implies 
 $$
   \cI_{W'}((\ell-\beta)E)\subset \cI_Z
 $$
  and therefore
 $\ell-\beta \le 0$. The case $\ell-\beta=0$ gives an inclusion $\cI_{W'}\subset \cI_Z$ and is treated using a diagram analogous to \eqref{E:extn3},     leading to a contradiction as before.

 If $\ell-\beta < 0$ then $\beta \ge \ell+1$. If $\beta > \ell+1$ then \eqref{E:extn2} gives a contradiction. If $\beta=\ell+1$ then \eqref{E:extn2} gives:
 $$
 \deg(W)+\deg(W')=2
 $$
 If $\deg(W)>0$ then $A\otimes\cI_W=\sigma^*L^{-1}((\ell+1)E)\otimes\cI_W \subset \sigma^*L^{-1}((\ell+1)E)$ and therefore $\phi(\sigma^*L^{-1}((\ell+1)E))$ is a torsion subsheaf of $\E(E)$, a contradiction.  Then $W=\emptyset$ and $W'=Z$.  This gives $\cI_Z\subset \E(E)$, viz. $\cI_Z(-E)\subset \E$. This implies that the pullback
 $$
 \theta:\Ext^1(\cI_Z,E-H) \longrightarrow \Ext^1(\cI_Z(-E),E-H)  
 $$
 maps \eqref{E:extn1} to zero, thus the dual map:
 $$
 \theta^\vee: H^1(\wt S,\cI_Z\otimes(H-E)) \longrightarrow H^1(\wt S,\cI_Z\otimes H)
 $$
 is not surjective. But $\coker(\theta^\vee)\subset H^1(E,\cI_Z\otimes\cO_E(H))=0$ because $\cI_Z\otimes\cO_E(H)$ is an invertible sheaf of degree $\ge 0$ on $E$. We have a contradiction and the theorem is proved.
\end{proof}

\begin{rem}\rm
(i) For the first values of $\ell$ the condition of the theorem gives:
\begin{eqnarray*}
\ell = 1 &:& g\ge 5 \\
\ell = 2 &:& g \ge 7 \\
\ell = 3 &:& g\ge 11.
\end{eqnarray*}
\medskip

(ii)
The case $\ell=1$, $g=5$   has already been considered in \cite{iB95}.

\medskip 

(iii) An interesting implicit consequence of Theorem \ref{veryample}  is the following existence result:

\medskip

\emph{Under the assumptions of Theorem \ref{veryample} for a given $\ell \ge 2$, through each point $x \in S$  there exist integral curves in $\vert L\vert$ having an ordinary multiple point of multiplicity exactly $\ell$ at $x$ and no other singularities.}

\medskip

(iv) One can combine the main result of \cite{GL86}  with Theorem 1.4 of  \cite{FKP07} to deduce, in certain ranges of $g,\ell$, that $\sigma^*L(-\ell E)$ is not only very ample but even embeds $\wt S$ as an arithmetically Cohen-Macaulay surface. This is the case e.g. for $\ell=1$, $g \ge 5$ and for $\ell=2$ and $g \ge 9$.

\end{rem}

In the next section we are going to apply Theorem \ref{veryample} in the case $\ell=2$.


\section{Proof of Theorem \ref{main}}
  Let $x\in S$ be the singular point of $f(C)$, let $\sigma:\wt S \to S$ be the blow-up at $x$ and $E\subset \wt S$   the exceptional curve. Then $C \in \vert\sigma^*L(-2E)\vert$ and Theorem \ref{veryample} implies that $\sigma^*L(-2E)$ is very ample,   thus $C$ is a hyperplane section of $\wt S\subset \bP^{g-2}$ embedded by $\sigma^*L(-2E)$. Let $T:=f^*(x)=P+Q= C\cdot E$.  Then 
  $$
  \omega_C(-T)= (C+E)\cdot C-C\cdot E= C\cdot C
  $$ 
Therefore $\Phi_{\omega_C,\omega_C(-T)}$ is not surjective, by Theorem \ref{l'vovski}. Now we use Theorem 1.4 of \cite{FKP07}. Since $g \ge 9$ we have, in the notation of \cite{FKP07}:
$$
\rho_{sing}(g,1,4,g-1)= \rho(g-1,1,4)+1 < 0
$$
and
$$
\rho_{sing}(g,2,6,g-1)= \rho(g-1,2,6)+1 < 0
$$
Therefore by \cite{FKP07}, Theorem 1.4, we have  $W^1_4(C)=\emptyset = W^2_6(C)$, thus  $\mathrm{Cliff}(C) \ge 3$. We can then apply Theorem \ref{comparison} to deduce that $\Phi_{\omega_C(-T)}$ is not surjective either. \qed


\end{document}